\documentclass[11pt]{article}

\usepackage{amsmath,amscd,amsfonts,eucal,latexsym,amssymb,mathrsfs}
\usepackage{amsxtra, amsthm, calc, graphics, color}
\theoremstyle{definition}
\newtheorem{Definition}{Definition}[section]
\theoremstyle{plain}
\newtheorem{Theorem}[Definition]{Theorem}
\newtheorem{Proposition}[Definition]{Proposition}
\newtheorem{Lemma}[Definition]{Lemma}
\newtheorem{Corollary}[Definition]{Corollary}
\theoremstyle{remark}
\newtheorem{Remark}[Definition]{Remark}

\numberwithin{equation}{section}
\newcounter{remcount}

\def\cC{{\cal C}}
\def\cD{{\cal D}}

\def\cK{{\cal K}}
\def\cL{{\cal L}}

\def\bC{{\mathbb C}}

\def\bN{{\mathbb N}}

\def\bR{{\mathbb R}}

\def\l{\lambda}

\def\s{\sigma}
    
\def\fA{{\mathfrak A}}

\def\supp{{\text{supp}\,}}

\newcommand{\Id}{{\bf 1}}
\renewcommand{\Im}{{\mathrm{Im}\,}}
\renewcommand{\Re}{{\mathrm{Re}\,}}
\newcommand{\sgn}{{\mathrm{sgn}}}
\newcommand{\cCm}{\cC^{(m)}}

\newcommand{\Om}{\Omega^{(m)}}

\newcounter{propcount}

\newlength{\maxlabelwidth}

\newcommand{\inst}[1]{$^\textrm{#1}$ }

\begin{document}

\title{Quasi-free isomorphisms of second quantization algebras and modular theory}
\author{Roberto Conti\inst{1} \and Gerardo Morsella\inst{2}}
\date{
\parbox[t]{0.9\textwidth}{\footnotesize{%
\begin{itemize}
\item[1] Dipartimento di Scienze di Base e Applicate per l'Ingegneria, 
Sapienza Universit\`a di Roma, via A. Scarpa, 16 I-00161 Roma (Italy), e-mail: roberto.conti@sbai.uniroma1.it
\item[2] Dipartimento di Matematica, Universit\`a di Roma Tor Vergata, via della Ricerca Scientifica, 1 I-00133 Roma (Italy), e-mail: morsella@mat.uniroma2.it
\end{itemize}
}}
\\
\vspace{\baselineskip}
\today}

\maketitle

\begin{center}
\emph{Dedicated to Roberto Longo on the occasion of his 70th birthday}
\end{center}

\begin{abstract}
Using Araki-Yamagami's characterization of quasi-equivalence for quasi-free representations of the CCRs, we provide an abstract criterion for the existence of isomorphisms of second quantization local von Neumann algebras induced by Bogolubov transformations in terms of the respective one particle modular operators. We discuss possible applications to the problem of local normality of vacua of Klein-Gordon fields with different masses.
\end{abstract}

\section{Introduction}

Since the time of its appearence, 
Tomita-Takesaki modular theory has been recognized as a fundamental tool in the study of von Neumann algebras. 
The outstanding work of A. Connes, Takesaki, Haagerup and others on the classification of injective factors of type III is a notable illustration of this fact (see~\cite{Ta} for a detailed exposition). 
Not too surprisingly, modular theory plays also a major role in the algebraic approach to Quantum Field Theory (QFT), 
as first recognized by Bisognano-Wichmann in the middle of the seventies. Indeed it was shown under mild assumptions that the one-parameter group of modular unitaries associated to certain spacetime wedges implements the Lorentz boosts.
However, while many more related aspects have been subsequently clarified in later works by Borchers, Buchholz, Longo, Schroer, Wiesbrock and many others (see~\cite{Ha} for a review), the properties of the modular operators associated to bounded regions in the theory of a scalar massive free field remain elusive 
(see e.g.\ the recent discussion in \cite{BCM}). 

In the case of free field theories, the modular theory of local algebras has a counterpart at the one particle level. Namely, it is possible to associate to certain real subspaces (known as standard subspaces) of the one particle Hilbert space modular operators whose second quantization yields the modular operators of the von Neumann algebra associated to the given subspace. Within this setting, a further aspect of the relevance of modular theory in QFT, which was overlooked until very recently, shows up. Namely the polariser of a standard subspace, connecting its symplectic structure with its real Hilbert space one, can be expressed as a functional calculus of the corresponding modular operator~\cite{Lo} (see also Prop.\ \ref{prop:R} below).

The interest of this observation lies in the fact that the polariser is a central object in the analysis of the quasi-equivalence properties of quasi-free states on Weyl algebras, as was discussed independently by Araki, van Daele and others \cite{AS,Ar,VD, AY}, without explicit use of modular theory. The main object of the present work is then to characterize the quasi-free isomorphisms of second quantization von Neumann algebras in terms of their one particle modular structure. We accomplish this task in Sec.~\ref{sec:general} relying on the powerful quasi-equivalence criterion of~\cite{AY}. A related, but apparently different, criterion has appeared in~\cite{Lo}.

A natural application to QFT of such a result is the analysis of local quasi-equivalence of vacuum states of free Klein-Gordon fields of different (non negative) masses, that was discussed about fifty years ago by Eckmann and Fr\"ohlich \cite{EF}. However, rather than using modular theory, whose relevance for QFT at the time was just 
starting to be recognized,  their  proof is deeply rooted on the remarkable work of Glimm and Jaffe in constructive field theory (roughly speaking, their idea is to look
at the difference of the mass terms as an interaction). Unfortunately, the justification of some of its central statements, though probably clear for people working in the field at that time, is only sketched, and moreover the proof makes essential use of results by Rosen which were available only in preprint form, and which cannot be easily found at present. This makes in our opinion \cite{EF} difficult to understand for non experts in constructive methods nowadays. Furthermore, an independent proof using modern techniques based on modular theory would be desirable also from a conceptual point of view. 

In Sec.~\ref{subsec:highdim}, we then provide a direct verification of some of the conditions appearing in our abstract result of Sec.~\ref{sec:general} in  the case of the restriction to a local algebra of massive and massless Klein-Gordon vacua in $d=2,3$ spatial dimensions.
As a byproduct, we show that the resolvents of the local modular operators on the one-particle space depend continuously on the field mass with respect to the Hilbert-Schmidt norm of the relevant local standard subspace. 
This fact looks interesting on its own (cf.\ e.g.\ \cite{SN})
but also enriches the scenario around modular theory for Weyl algebras developed in \cite{Lo}.
Of course, strictly speaking our result is implied by the analysis in \cite{EF}, but our arguments provide a self-contained independent 
approach which is somewhat more intrinsic and also sheds some new light on the mathematical 
structures behind these operators. 
Unfortunately, a complete characterization of local quasi-equivalence of such states along these lines seems to require a more detailed knowledge of the massive modular operator than is presently available.

The case of a massive and a massless vacuum in $d=1$ spatial dimension is not covered by the results of Eckmann and Fr\"ohlich, due to an infrared divergence. In Sec.~\ref{subsec:lowdim} we show that our abstract quasi-equivalence criterion and the results of~\cite{CM3} imply the 
existence of a quasi-free
isomorphism between the local von Neumann algebras generated by massive and massless Weyl operators corresponding to test functions of zero mean in $d=1$ (see~\cite{BFR} for a related result).

\section{General results}\label{sec:general}
Let $(H,\langle\cdot,\cdot\rangle)$ be a complex Hilbert space (where the scalar product is linear in the second variable). We consider the usual associated symmetric Fock space
\[
e^H = \bigoplus_{n = 0}^{+\infty} H^{\otimes_S n}
\] 
in which the coherent vectors
\[
e^x := \bigoplus_{n=0}^{+\infty} \frac1{\sqrt{n!}} x^{\otimes n}, \qquad x \in H,
\]
form a total set (here $H^{\otimes_S 0} := \bC$ and $x^{\otimes 0} := 1$). In particular the vector $\Omega := e^0 \in e^H$ is called the vacuum vector.  We also consider the Weyl unitaries $W(x), x \in H$, on $e^H$, defined by their action on $\Omega$ and by the canonical commutation relations (CCR):
\begin{gather*}
W(x) e^0 := e^{-\frac14\|x\|^2} e^{i x/\sqrt{2}}, \qquad x \in H,\\
W(x) W(y) = e^{-\frac i 2 \Im\langle x,y\rangle} W(x+y), \qquad x,y \in H.
\end{gather*}

For any closed real subspace $K$ of $H$, we define a von Neumann algebra
\[
A(K) = \{W(h) \ | \ h \in K\}''
\] 
on $e^H$, the second quantized algebra of $K$.
We say that $K$ is  \emph{standard} if $K + iK$ is dense in $H$ and $K \cap iK = \{0\}$. Then the vacuum vector $\Omega \in e^H$ is cyclic and separating for $A(K)$ if and only if
$K$ is standard~\cite{Ar0}.
To any standard subspace $K$ we associate a closed, densely defined conjugate linear operator
\[
s: K + iK \to K + iK, \quad s(h + ik) = h - ik,  \qquad h,k \in K.
\]
Moreover, if $s= j \delta^{1/2}$ is the polar decomposition, we call $j$ and $\delta$ the modular conjugation and the modular operator  of $K$. They satisfy:
\[
j^* = j = j^{-1}, \quad j \delta = \delta^{-1}j, \quad jK = K', \quad \delta^{it}K = K, \;t \in\bR.
\]
The operators $J = \Gamma(j)$, $\Delta = \Gamma(\delta)$ are respectively the modular conjugation and the modular operator of $A(K)$ with respect to $\Omega$~\cite{EO,L2}, where for a closed densely defined operator $T : D(T) \subset H \to H$, its second quantization $\Gamma(T)$ is the closure of the operator on $e^H$ defined on the linear span of coherent vectors $e^x$, $x \in D(T)$, by $\Gamma(T)e^x := e^{Tx}$.  Also, one has $\log{\Delta} = d\Gamma(\log\delta)$, where $d\Gamma$ denotes the usual second quantization of a self-adjoint operator \cite[Sec.~10.7]{RS2}. Since  $s$ is invertible, $0$ is not an eigenvalue of $\delta$, and then, following~\cite{FG2}, we can also introduce the bounded self-adjoint operators $\theta$, $\gamma$ on $e^H$, defined as
$$\tan(\theta/2) = e^{-\frac{1}{2}|\log \delta|}, \qquad \gamma = \sgn \log \delta,$$
with the convention $\sgn(0):=0$. Notice that $\sigma(\theta) \subset [0,\pi/2]$, $[\gamma,\theta] = 0 = [j,\theta]$, and $j\gamma j = -\gamma$. 

The symplectic complement of the standard subspace $K$ is the standard subspace $K':= \{ h' \in H\,:\, \Im\langle k,h'\rangle = 0 \; \forall k \in K\}$, and $K$ is called a factor if $K \cap K' = \{0\}$, which is equivalent to $A(K)$ being a factor von Neumann algebra, namely $A(K) \cap A(K)' = {\mathbb C} \Id$. In this case, $K$ equipped with the non-degenerate form $\sigma(h,k):= \Im \langle h,k \rangle, h,k \in K$, becomes a (real) symplectic space.
Also, $K$ is a factor if and only if $1 \not \in \sigma_p(\delta)$.
Indeed, on the one hand $K \cap K' = \{h \in K \ | \ \delta h = h\}$ (cf. \cite{FG2}) and, on the other hand, if $0 \neq x \in H$ and $\delta x = x$ then $x + j x$ and $ix + j(ix)=i(x - j x)$ are both invariant under $\delta$, $j$ and $s$, therefore belong to $K \cap K'$ and at least one of them is nonzero.
\medskip

The following proposition shows that the modular structure of $K$ relates its real Hilbert space structure with the symplectic one, cf.~\cite[Prop.\ 2.4]{Lo}.

\begin{Proposition}\label{prop:R}
Let $K$ be a standard real subspace of the Hilbert space $H$,
and consider the bounded  skew-adjoint operator $
R:= i \frac{\delta - 1}{\delta + 1}$ on $H$. Then it holds:
\renewcommand{\theenumi}{(\roman{enumi})}
\renewcommand{\labelenumi}{\theenumi}
\begin{enumerate}
\item $R h \in K$  
for every $h \in K$;
\item $-R^2 \leq 1$;
\item \label{it:imre}$ \Im\langle h,k \rangle = \Re \langle h, Rk \rangle, \ h, k \in K \ ;$
\end{enumerate}
\end{Proposition}
\begin{proof}
(i) Since for $y \in H$ one has $j\delta^{\frac 12}y=y$ if and only if $y \in K$, given $h \in K$ the statement follows from
\[
j\delta^{\frac 12}Rh = ij\frac{1-\delta}{1+\delta}\delta^{\frac12} h = i\frac{1-\delta^{-1}}{1+\delta^{-1}}j\delta^{\frac 1 2}h = Rh.
\]

(ii) It is an immediate consequence of the fact that the function $\l \in [0,+\infty) \mapsto (\frac{\l-1}{\l+1})^2$ is bounded above by 1.

(iii) We begin by observing that, if $e_-, e_1, e_+$ denote the spectral projections of $\delta$ associated to the sets $[0,1), \{1\}, (1,+\infty)$ respectively, then $\gamma = e_+-e_-$ and
\[
i\gamma\cos \theta = i\frac{1-\tan^2(\theta/2)}{1+\tan^2(\theta/2)}(e_+-e_-) = i \frac{1-\delta^{-1}}{1+\delta^{-1}}e_+- i \frac{1-\delta}{1+\delta}e_- = R,
\]
where the last equality holds since $Re_1 = 0$.
Now given $h \in K$, with $h_\pm := e_\pm h$, one has
\[
\tan(\theta/2) j h_- = \delta^{-\frac 12}e_+ jh = e_+ j\delta^{\frac 12 }h = h_+
\]
and therefore, if $h,k \in K$,
\[\begin{split}
\Re\langle h,Rk\rangle &= \frac i2\left[\left\langle h_++h_-, \cos \theta(k_+-k_-) \right\rangle-\left\langle  \cos \theta(k_+-k_-),  h_++h_-\right\rangle\right]\\
&= \frac i 2\left[\left\langle j\tan(\theta/2)h_-,j\tan(\theta/2)\cos \theta k_-\right\rangle - \left\langle h_-,\cos\theta k_-\right\rangle\right. \\
&\qquad -\left.\left\langle j\tan(\theta/2)\cos\theta k_- ,j\tan(\theta/2)h_-\right\rangle + \left\langle \cos\theta k_-,h_-\right\rangle\right] \\
&=\frac i2\left[ \left\langle k_-,(\tan^2(\theta/2)+1){\cos\theta} h_-\right\rangle- \left\langle (\tan^2(\theta/2)+1)\cos\theta h_-,k_-\right\rangle\right]\\
&=-\frac i2\left[ \left\langle k_-,(\tan^2(\theta/2)-1)h_-\right\rangle- \left\langle (\tan^2(\theta/2)-1)h_-,k_-\right\rangle\right].
\end{split}\]
On the other hand $h_1 := e_1h, k_1 := e_1 k \in K \cap K'$ entails $\Im\langle h_1,k_1\rangle = 0$, and therefore 
\[\begin{split}
\Im \langle h,k\rangle &= -\frac i2\left[ \langle h_++h_-,k_++k_-\rangle - \langle k_++k_-,h_++h_-\rangle\right] \\
&= -\frac i 2\left[ \langle j \tan(\theta/2)h_-,j\tan(\theta/2)k_-\rangle +\langle h_-,k_-\rangle -  \langle j \tan(\theta/2)k_-,j\tan(\theta/2)h_-\rangle -\langle k_-,h_-\rangle\right]\\
&=-\frac i2\left[ \left\langle k_-,(\tan^2(\theta/2)-1)h_-\right\rangle- \left\langle (\tan^2(\theta/2)-1)h_-,k_-\right\rangle\right],
\end{split}\]
thus proving the required formula.
\end{proof}

We collect some further properties of the operator $R$ introduced above in the  factorial case.

\begin{Proposition}\label{prop:Rprop}
Let $K$ be a standard factorial subspace of $H$. Then:
\renewcommand{\theenumi}{(\roman{enumi})}
\renewcommand{\labelenumi}{\theenumi}
\begin{enumerate}
\item\label{it:factor} the operator $R^{-1} = - i\frac{\delta+1}{\delta-1}$ is a densely defined skew-adjoint operator (possibly unbounded) on $H$, mapping the dense subspace $RK$ of $K$ onto $K$;
\item\label{it:Rs} the Tomita operator $s$ maps $D(R^{-1}) \cap (K+iK)$ into $D(R^{-1})$;
\item  if $1 \not \in \s(\delta)$, then $RK = K$.
\end{enumerate}
\end{Proposition}
\begin{proof}
(i) We only need to check the density of $RK$ in $K$. If $K$ is a factor (i.e., $1 \not \in \sigma_p(\delta)$), denote by $e_n$ the spectral projection of $\delta$ relative to the set $[0,1-1/n)\cup((1-1/n)^{-1},+\infty)$, $n \in \bN$. Therefore, given $k \in K$, the vectors $e_n k$ are in $D(R^{-1})$. Moreover,  since $s R^{-1} = R^{-1} s$ and $s e_n = e_n s$ by functional calculus, $R^{-1}e_n k \in K$. Finally, $e_n k = R(R^{-1}e_n k)$ converge to $k$.

(ii) Let $x \in H$ be such that $Rx \in K+iK = D(s)$, then $j Rx$ is in the domain of $\delta^{-1/2}$ and, again by functional calculus, $sRx = \delta^{-1/2} jRx = \delta^{-1/2} R jx = R \delta^{-1/2} jx$ is in the domain of $R^{-1}$.

(iii) The operator $R^{-1} = i\frac{1+\delta}{1-\delta}$ is bounded and commutes with $s= j \delta^{1/2}$, therefore $R^{-1} K \subset K$, i.e., $K \subset RK$.
\end{proof}

Given a standard subspace $K \subset H$, we can introduce on $K +iK$ the scalar product inducing the graph norm of $s$:
\[
\langle x,y\rangle_s:= \langle x,y\rangle + \langle sy,sx\rangle  \qquad x,y \in K+iK. 
\]
Obviously, since $s$ is closed, $K+iK$ is a complex Hilbert space with respect to this scalar product. Notice also that for $h_j, k_j \in K$, $j=1,2$,
\begin{equation}\label{eq:graphprod}
\langle h_1+ih_2, k_1+i k_2\rangle_s = 2 \Re\left[\langle h_1,k_1\rangle + \langle h_2,k_2\rangle\right]+2i\Re\left[\langle h_1,k_2\rangle-\langle h_2,k_1\rangle\right].
\end{equation}

We observe that $\frac1{1+\delta} = \frac12(1+iR)$ maps $K+iK$ into itself, and it is easy to verify that it is a bounded self-adjoint operator on the Hilbert space $(K+iK, \langle \cdot, \cdot\rangle_s)$. Moreover, if $f : [0,1] \to \bR$ is a bounded continuous function, by uniform approximation with polynomials the functional calculus $f(\frac 1{1+\delta})$ maps $K+iK$ into itself, and its restriction to $K+iK$ coincides with the functional calculus of $f$ on $\frac1{1+\delta}$ thought as an operator on  $(K+iK, \langle \cdot, \cdot \rangle_s)$. In particular, if $f$ is non negative, $f(\frac1{1+\delta})$ is a positive operator on $K+iK$. Thus, in the following, it will not be necessary to specify what is the Hilbert space structure with respect to which such a functional calculus is considered.
\medskip

Given now two standard subspaces $K_1, K_2 \subset H$ and a real linear bounded operator $T : K_1 \to K_2$, we denote by $T^\dagger : K_2 \to K_1$ the adjoint of $T$ w.r.t.\ the real Hilbert space structures of $K_1, K_2$, namely $T^\dagger$ is the unique real linear bounded operator satisfying
\[
\Re \langle T^\dagger h,k\rangle = \Re \langle h, Tk\rangle, \qquad h \in K_2,k \in K_1.
\]

We remark, for future reference, that given a linear operator $T : K_1 + i K_1 \to K_2 + i K_2$ which is bounded with respect to the scalar products $\langle \cdot, \cdot\rangle_{s_j}$, $j=1,2$, and such that $T(K_1) \subset K_2$, with the notation $\hat T := T|_{K_1}$ one has, for all $k \in K_1$, $h_1,h_2 \in K_2$, 
\[\begin{split}
\langle h_1+ih_2, Tk\rangle_{s_2} &= 2 \Re \langle h_1, \hat T k\rangle -2i\Re\langle h_2, \hat T k\rangle\\
&= 2 \Re\langle \hat T^\dagger h_1, k\rangle -2i \Re\langle \hat T^\dagger h_2, k\rangle = \langle \hat T^\dagger h_1+i\hat T^\dagger h_2, k \rangle_{s_1}.
\end{split}\]
Then the adjoint of $T$ coincides with the $\bC$-linear extension of $\hat T^\dagger$ to $K_2+iK_2$, and therefore it will be denoted also by $T^\dagger$. Moreover, this entails $T^\dagger T|_{K_1} = \hat T^\dagger \hat T$ and since, by functional calculus, $(T^\dagger T)^{1/2}$ maps $K_1$ into itself, and
\[
\Re \langle k, (T^\dagger T)^{1/2}k\rangle = \frac12 \langle k, (T^\dagger T)^{1/2}k\rangle_{s_1} \geq 0, \qquad k \in K_1,
\]
we conclude that the restriction to $K_1$ of $(T^\dagger T)^{1/2}$ coincides with $(\hat T^\dagger \hat T)^{1/2}$, namely
\begin{equation}\label{eq:absolute}
|T|\,|_{K_1} = |\hat T|.
\end{equation}

We now introduce a natural notion of equivalence between standard subspaces.
\begin{Definition}
Let $K_1, K_2 \subset H$ be standard subspaces. A  real linear bijection $Q: K_1 \to K_2$ is said to be a (real) \emph{symplectomorphism} if $\Im\langle Qh,Qk\rangle = \Im\langle h,k\rangle$ for all $h,k \in K_1$.
\end{Definition}

Clearly, if $Q: K_1 \to K_2$ is a symplectomorphism, the operators $W(Qk)$, $k \in K_1$, still satisfy the CCRs, so that they generate a C*-algebra isomorphic to the one generated by the $W(k)$, $k \in K_1$. It is then natural to ask for which $Q$ such isomorphism extends to an isomorphism between the corresponding second quantization von Neumann algebras $A(K_1)$ and $A(K_2)$.

\begin{Definition}
Let $K_1, K_2 \subset H$ be standard subspaces, and $Q : K_1 \to K_2$ a symplectomorphism. An isomorphism of von Neumann algebras $\phi : A(K_1) \to A(K_2)$ such that $\phi(W(k)) = W(Qk)$, $k \in K_1$, is called the \emph{quasi-free (or Bogolubov) isomorphism induced by $Q$}.
\end{Definition}

The following result characterizes those symplectomorphisms between standard subspaces which induce quasi-free isomorphisms between the corresponding second quantization von Neumann algebras in terms of their one particle modular structures.

\begin{Theorem}\label{main}
Let $K_1, K_2 \subset H$ be standard subspaces, and $Q: K_1 \to K_2$ be a (real) symplectomorphism. There exists an isomorphism $\phi: A(K_1) \to A(K_2)$ such that $\phi(W(h)) = W(Qh)$, $h \in K_1$, if and only if:
\renewcommand{\theenumi}{(\roman{enumi})}
\renewcommand{\labelenumi}{\theenumi}
\begin{enumerate}
\item $Q$ is bounded (w.r.t.\ the restriction to $K_1$ and $K_2$ of the norm $\|\cdot\|$ on $H$);
\item the operator
\begin{equation}\label{eq:HS}
\big[1+i R_1\big]^{\frac 12}-\big[ i  R_1 + Q^\dagger Q\big]^{\frac 12}
\end{equation}
is Hilbert-Schmidt on $K_1 + i K_1$.
\end{enumerate}
\end{Theorem}

\begin{proof}
The statement is a direct application of~\cite{AY}. Indeed, defining the indefinite inner product $\gamma : (K_1+iK_1) \times (K_1+iK_1) \to \bC$
\[
\gamma(x,y) := \langle x,y\rangle - \langle s_1y,s_1x\rangle, \qquad x,y \in K_1+iK_1,
\]
one verifies at once that $\gamma(s_1x,s_1y) = -\gamma(y,x)$, and therefore one can consider the self-dual CCR algebra $\fA(K_1+iK_1,\gamma,s_1)$. Moreover one computes
\[
\gamma(h_1+ih_2,k_1+ik_2) = 2i\Im[\langle h_1,k_1\rangle + \langle h_2,k_2\rangle] - 2 \Im[\langle h_1,k_2\rangle - \langle h_2,k_1\rangle],
\]
which easily entails $\langle Qx,Qy\rangle -\langle s_2Qy,s_2Qx\rangle = \gamma(x,y)$, $x,y \in K_1+iK_1$, where $Q$ is extended to $K_1+iK_1$ by $\bC$-linearity (thus getting a bijection onto $K_2 +i K_2$). Therefore, if $S, S': (K_1+iK_1)\times (K_1+iK_1) \to \bC$ are the hermitian forms
\[
S(x,y) := \langle x,y\rangle, \qquad S'(x,y) :=  \langle Qx,Qy\rangle, \qquad x,y \in K_1+iK_1,
\] 
one finds $S(x,y) - S(s_1y,s_1x) = \gamma(x,y) = S'(x,y)-S'(s_1y,s_1x)$, $x,y \in K_1+iK_1$, which implies that $S$, $S'$ define quasi-free states $\varphi, \varphi'$ on $\fA(K_1+iK_1,\gamma,s_1)$ such that $\varphi(x^*y) = S(x,y)$ (and similarly for $\varphi'$, $S'$). Now clearly
\begin{align*}
\langle x,y\rangle_S &:= S(x,y)+S(s_1y,s_1x) = \langle x,y\rangle_{s_1},\\
\langle x,y\rangle_{S'} &:= S'(x,y)+S'(s_1y,s_1x) = \langle Qx,Qy\rangle_{s_2},
\end{align*}
and, by Eq.~\eqref{eq:graphprod},
\[
\| k_1+ik_2\|^2_{s_1} = 2(\|k_1\|^2+\|k_2\|^2), \qquad \| Qk_1+iQk_2\|^2_{s_2} = 2(\|Qk_1\|^2+\|Qk_2\|^2).
\]
Then the boundedness (with respect to the norm $\|\cdot\|$ on $H$) of $Q$, which implies, by the open mapping theorem, that of its inverse, is equivalent to the fact that $\langle \cdot, \cdot\rangle_S$ and $\langle \cdot,\cdot\rangle_{S'}$ define the same topology on $K_1+iK_1$. Moreover, extending also $Q^\dagger$ to $K_2 +i K_2$ by $\bC$-linearity, 
by Eq.~\eqref{eq:graphprod} and Prop.~\ref{prop:R}\ref{it:imre} we get
\begin{equation}\label{eq:R-S}
\Big\langle x,\frac1 2\big[1+i R_1\big]y\Big\rangle_{s_1} = S(x,y), \qquad \Big\langle x, \frac 1 2\big[iR_1+Q^\dagger Q\big] y\Big\rangle_{s_1} = S'(x,y).
\end{equation}
Then by the main result of~\cite{AY} (Theorem on p.\ 285) the existence of the isomorphism $\phi : A(K_1) \to A(K_2)$ is equivalent to the operator $[1 + iR_1]^{\frac12}- [iR_1 + Q^\dagger Q]^{\frac12}$ being Hilbert-Schmidt on $K_1 + i K_1$ equipped with the scalar product $\langle \cdot, \cdot \rangle_{s_1}$. 
\end{proof}
It is worth pointing out that  thanks to~\eqref{eq:graphprod}, the restriction of $\langle \cdot, \cdot\rangle_{s_1}$ to $K_1$ is (a multiple of) the real scalar product $\Re\langle \cdot,\cdot \rangle$ of $K_1$. This implies that an orthonormal basis of the real Hilbert space $K_1$ is also an orthonormal basis of the complex Hilbert space $K_1+iK_1$ (up to a common normalization), and therefore (also thanks to eq.~\eqref{eq:absolute}) an operator on $K_1 + i K_1$ mapping $K_1$ into itself is Hilbert-Schmidt (resp.\ trace class) on $K_1 + i K_1$ if and only if its restriction is Hilbert-Schmidt (resp.\ trace class) on $K_1$.
Then, in particular, by~\cite[Lemma\ 4.1]{PS} a sufficient condition for (ii) is that
\[
\big[1+ i  R_1\big]-\big[ i  R_1 + Q^\dagger Q\big]= 1 - Q^\dagger Q
\]
is trace class on $K_1$. On the other hand, given bounded operators $A_j, B_j$ on $K_1$ such that $A_j - B_j$ is in a Schatten class, $j=1,2$, then by the identity
\begin{equation}\label{eq:usual}
A_1 B_1 - A_2 B_2 = A_1 (B_1-B_2) + (A_1-A_2) B_2
\end{equation}
one sees that $A_1 B_1 - A_2 B_2$ is again in the same Schatten class. Therefore, a necessary condition for (ii) is that $1 - Q^\dagger Q$ is Hilbert-Schmidt on $K_1$. Moreover, it follows easily from~\cite[Prop.\ 6.6(iv)]{AY}, by arguments similar to those in the proof of the above theorem, that another condition equivalent to (ii) is that the operators
\begin{equation}\label{eq:equiv}
1 - Q^\dagger Q, \qquad \sqrt{\frac12(1 +iR_1)}- Q^{-1}\sqrt{\frac12(1 +iR_2)}Q = \frac{1}{\sqrt{1+\delta_1}} - Q^{-1} \frac{1}{\sqrt{1+\delta_2}} Q
\end{equation}
are both Hilbert-Schmidt on $K_1 
+ i K_1$.
(Notice that the second operator is not necessarily self-adjoint on $K_1 + i K_1$.)

\begin{Corollary}\label{factor}
Let $K_1, K_2 \subset H$ be factorial standard subspaces, and $Q: K_1 \to K_2$ be a symplectomorphism. There exists an isomorphism $\phi: A(K_1) \to A(K_2)$ such that $\phi(W(h)) = W(Qh)$, $h \in K_1$, if and only if:
\begin{align}\label{eq:HSfactor}
\big[1+i R_1\big]^{\frac 12}-\big[ i  R_1 + R_1Q^{-1} R_2^{-1} Q\big]^{\frac 12} & = 
\left[1 + \frac{1-\delta_1}{1+\delta_1}\right]^{\frac 12}-\left[\frac{1-\delta_1}{1+\delta_1}+\frac{1-\delta_1}{1+\delta_1}Q^{-1}\frac{1+\delta_2}{1-\delta_2}Q\right]^{\frac 12}
\end{align}
is Hilbert-Schmidt on $K_1+iK_1$, where, by a slight abuse of notation, the unique bounded extension to $K_2+iK_2$ of $R_1 Q^{-1} R_2^{-1} : R_2 (K_2+iK_2) \to K_1+iK_1$ (which exists in the above hypotheses) is denoted by the same symbol.
\end{Corollary}

\begin{proof}
We start by observing that the operator~\eqref{eq:HSfactor} is a priori densely defined and bounded. Indeed, for all $h \in K_1$, $k\in R_2 K_2$, there holds
\begin{equation}\label{eq:Qfactor}\begin{split}
\Re\langle Qh, k\rangle &= \Re\langle Qh, R_2 R_2^{-1}k\rangle = \Im \langle Qh, R_2^{-1}k\rangle\\
&= \Im \langle h, Q^{-1}R_2^{-1}k\rangle = \Re\langle h, R_1 Q^{-1} R_2^{-1} k\rangle.
\end{split}\end{equation}
Thanks to Prop.~\ref{prop:Rprop}\ref{it:factor} and to~\eqref{eq:graphprod}, this implies that, in the unbounded operator sense, $Q^\dagger \supset R_1 Q^{-1} R_2^{-1}$ is densely defined as an operator from $K_2+i K_2$ to $K_1+i K_1$, so that $Q$ is closable and, being everywhere defined, it is closed. Thus, by polar decomposition and the Hellinger-Toeplitz theorem, it is bounded. Moreover $R_1 Q^{-1} R_2^{-1}$ and $R_1 Q^{-1} R_2^{-1}Q$ extend uniquely to the bounded operators $Q^\dagger$ and $Q^\dagger Q$. The statement then follows at once from Thm.~\ref{main}.
\end{proof}

We notice also that if $K_j$ is a factor and $\delta_j$ is bounded 
 on $H$ 
(or, equivalently, $0 \not\in \s(\delta_j$)), $j=1,2$, then for the existence of the quasi-free isomorphism $\phi : A(K_1) \to A(K_2)$ it is actually sufficient that $1 - Q^\dagger Q$ is Hilbert-Schmidt on $K_1$ (equivalently, on $K_1 + i K_2$ w.r.t. $\langle \cdot,\cdot \rangle_{s_1}$). Indeed, 
\[
1-Q^\dagger Q = \left[1+\frac{1-\delta_1}{1+\delta_1}\right] - \left[\frac{1-\delta_1}{1+\delta_1}+Q^\dagger Q\right],
\]
and, from the boundedness of $\delta_1$, we see that $0$ does not belong to the spectrum of $1+\frac{1-\delta_1}{1+\delta_1} = \frac2{1+\delta_1}$ on $K_1 + i K_1$. This follows from the identity, valid for $x \in K_1+iK_1$,
\[
\left\langle x, \frac1{1+\delta_1}x \right\rangle_{s_1} = \left\langle x,  \frac1{1+\delta_1}x \right\rangle +\left\langle s_1  \frac1{1+\delta_1}x ,s_1x \right\rangle = \left\langle x, (1+\delta_1) \frac1{1+\delta_1}x\right\rangle = \|x\|^2,
\]
which, thanks to the boundedness of $\delta_1$, implies
\[
\|x\|_{s_1}^2 = \langle x, x\rangle_{s_1} = \langle x, (1+\delta_1) x\rangle \leq (1+\|\delta_1\|)\|x\|^2 = (1+\|\delta_1\|)\left\langle x, \frac1{1+\delta_1} x\right\rangle_{s_1},
\]
i.e., $ \frac1{1+\delta_1} $ is bounded below away from $0$ on $K_1+iK_1$.
Moreover
\[\begin{split}
\frac{1-\delta_1}{1+\delta_1}+Q^\dagger Q &= Q^\dagger \left((Q^\dagger)^{-1} \frac{1-\delta_1}{1+\delta_1}+Q\right)
= Q^\dagger\left(\frac{1-\delta_2}{1+\delta_2} Q + Q\right)= Q^\dagger \frac{2}{1+\delta_2} Q
\end{split}\]
is also invertible with bounded inverse on $K_1 + iK_1$, i.e., its spectrum does not contain $0$ either, so we can apply~\cite[Thm.~4.1]{GPS} with $f(\lambda) = \lambda^{1/2}$ and obtain that~\eqref{eq:HSfactor} is Hilbert-Schmidt on $K_1+iK_1$.

\begin{Remark}\label{rem:auto}
Given a standard (factorial) subspace $K$, the quasi-free automorphisms $\phi : A (K) \to A(K)$ clearly form a subgroup of the automorphism group of $A(K)$. The above result then implies that the same is true for the bounded symplectomorphisms $Q : K \to K$ such that 
\begin{equation}
\label{symauto}
\left[1 + \frac{1-\delta}{1+\delta}\right]^{\frac 12} - \left[\frac{1-\delta}{1+\delta}+Q^\dagger Q\right]^{\frac 12} 
\end{equation}
is Hilbert-Schmidt on $K + iK$.
\end{Remark}

In the factorial case, a criterion for the existence of quasi-free isomorphisms related to the one in Thm.~\ref{main} could be obtained by using, instead of~\cite{AY}, the results of~\cite{VD}, but, due to the somewhat different assumptions of the latter work, a direct comparison between the two is not completely straightforward.  

The main issue is the fact that, in order to be able to apply~\cite{VD} to our setting, we would need to assume that $Q : K_1 \to K_2$ maps $C^\infty(A_1) \cap K_1$ onto $C^\infty(A_2) \cap K_2$, where $A_j:= R_j^{-1} = - i \frac{\delta_j + 1}{\delta_j-1}$, $j=1,2$. Such an hypothesis appears problematic from the point of view of the applications to QFT that we have in mind (see next Section).

 Another possibility would be to require that $1 \not \in \sigma(\delta_j)$  (which is also not true in QFT~\cite[end of Sec.\ 3]{FG}), so that $A_j$ is everywhere defined on $K_j$, $j=1,2$.
The necessary and sufficient condition for the quasi-equivalence of quasi-free states of the Weyl algebra found in~\cite{VD}
could then be rephrased in our language by saying that $Q$ induces a quasi-free automorphism $\phi : A(K_1) \to A(K_2)$ if and only if
\begin{equation}\label{eq:vD}
1- \frac{1-\delta_1^{1/2}}{1+\delta_1^{1/2}}Q^{-1} \frac{1+\delta_2^{1/2}}{1-\delta_2^{1/2}} Q   
= 1 - \tanh\left(\frac14 \log \delta_1\right) Q^{-1}\coth\left(\frac14 \log \delta_2\right) Q
\end{equation}
is Hilbert-Schmidt on $K_1$. We notice explicitly that it is not too difficult to provide a direct argument showing that if~\eqref{eq:vD} is Hilbert-Schmit on $K_1$ then the same is true for~\eqref{eq:HSfactor}.

\begin{Remark}
In~\cite[Thm.\ 4.9]{Lo}  a necessary and sufficient condition on $Q$ for the innerness of the associated quasi-free automorphism is given. Such condition then of course implies the validity of the condition in Rem.~\ref{rem:auto}, although a direct proof seems not to be immediate. In particular, a sufficient condition for the innerness is that $(Q-1)\frac{1+\delta}{1-\delta}i$ is Hilbert-Schmidt on $K$. By the boundeness of $\frac{1-\delta}{1+\delta}$ this entails that $Q-1$, and then $Q^\dagger -1$, are Hilbert-Schmidt on $K$, and thus on $K+iK$.  This in turn, by the identity~\eqref{eq:usual}, implies that $1-Q^\dagger Q$ is Hilbert-Schmidt on $K+iK$, and, by the boundedness of $Q^{-1}$ and of $\frac{1}{\sqrt{1+\delta}}$,
\[
\frac{1}{\sqrt{1+\delta}}- Q^{-1}\frac{1}{\sqrt{1+\delta}}Q = Q^{-1}\left[\frac{1}{\sqrt{1+\delta}}(1-Q)-(1-Q)\frac{1}{\sqrt{1+\delta}}\right]
\]
is Hilbert-Schmidt on $K+iK$ too.
\end{Remark}

\begin{Remark} In the setting of Fock states/representations, a result by Shale~\cite{Sh} shows that given a complex Hilbert space $H$ and setting $\Re H:= H$ as a real Hilbert space with its natural real inner product, the quasi-free automorphism of the associated Weyl $C^*$-algebra 
corresponding to $T \in {\rm Sp}(\Re H):=\{T \in B(\Re H) \ | \ T(\Re H) = \Re H, \ \Im\langle Tx,Ty\rangle = \Im\langle x,y\rangle, x,y \in \Re H\}$ is unitarily implemented if and only if $|T|-1$ is Hilbert-Schmidt on $\Re H$.
Note however that the representation of the Weyl C*-algebra associated to the symplectic space $(K,\Im \langle\cdot,\cdot\rangle)$ provided  by the Weyl operators $W(k)$, $k \in K$, on $e^H$ is not a Fock one, as $A(K) \neq B(e^H)$. 
\end{Remark}

\section{Applications to QFT}
In this section we discuss some applications of the above general framework to second quantization von Neumman algebras arising in the description of free field models in the algebraic approach to quantum field theory.

\subsection{Free scalar field in $d>1$}\label{subsec:highdim}
In $H = L^2(\bR^d)$, $d=2,3$, for a given mass $m \geq 0$ consider the standard factor subspace~\cite{Ar1, FG}
\begin{equation}\label{kappaemme}
K_m = \omega_{m}^{-\frac12} H_\bR^{-\frac12}(B) + i \omega_{m}^{\frac12}H_\bR^{\frac12}(B), 
\end{equation}
where $\omega_m^s$ is the multiplication operator by $\omega_m(p)^s = (m^2+p^2)^{s/2}$ in Fourier transform, $B \subset \bR^d$ is the ball of radius 1 centered around the origin, and $H_\bR^s(B)$ is the closure of $C^\infty_c(B,\bR)$ in the Sobolev space of real tempered distributions $f$ such that
\begin{equation}\label{eq:sobolev}
\int_{\bR^d} dp\, (1+p^2)^{s}|\hat f(p)|^2 < +\infty.
\end{equation}

\begin{Proposition}\label{prop:sunbound}
The one particle Tomita operator $s_m : K_m + i K_m \to K_m + i K_m$ is unbounded on $H$.
\end{Proposition}

\begin{proof}
We first notice that $K_m + i K_m =  \omega_{m}^{-\frac12} H^{-\frac12}(B) + i \omega_{m}^{\frac12}H^{\frac12}(B)$ with $H^s(B) = H_\bR^s(B)+i H^s_\bR(B)$ the Sobolev space of complex tempered distributions such that~\eqref{eq:sobolev} holds, and that
\[
s_m(\omega_m^{-\frac12} f + i \omega_m^{\frac12} g) = \omega_m^{-\frac12} \bar f + i \omega_m^{\frac12} \bar g, \qquad f \in H^{-1/2}(B), \, g \in H^{1/2}(B).
\]
Consider now a sequence $\{f_n \} \subset H^{-1/2}(B)$ converging to some $f \in H^{-1/2}(B)$, and set $g_n := i \omega_m^{-1} f_n \in H^{1/2}(B)$, $n \in \bN$. Then
\begin{gather*}
\| \omega_m^{-\frac12} f + i \omega_m^{\frac12} g_n\| = \| \omega_m^{-\frac12}(f-f_n)\| \to 0, \\
\| s_m( \omega_m^{-\frac12} f + i \omega_m^{\frac12} g_n)\| = \|  \omega_m^{-\frac12} \bar f + i \omega_m^{\frac12} \bar g_n\| =  \| \omega_m^{-\frac12}(\bar f+\bar f_n)\| \to 2 \| \omega_m^{-\frac12}f\|,
\end{gather*}
thus proving the statement.
\end{proof}

Given now two masses $m_1, m_2 \geq 0$, the operator $Q : K_{m_1} \to K_{m_2}$ defined by
\[
Q: \omega_{m_1}^{-\frac12} f + i \omega_{m_1}^{\frac12}g \mapsto \omega_{m_2}^{-\frac12} f + i \omega_{m_2}^{\frac12}g, \qquad f \in H_\bR^{-\frac12}(B), g \in H_\bR^{\frac12}(B),
\]
is a symplectomorphism which of course depends on both $m_1$ and $m_2$, but we will omit to indicate this explicitly in order to simplify the notation. By the result of Eckmann-Fr\"ohlich~\cite{EF}, $Q$ induces an isomorphism of the corresponding second quantization algebras. Therefore, by corollary~\ref{factor}, condition~\eqref{eq:HSfactor} holds for the modular operators $\delta_{m_j}$ associated to $K_{m_j}$. 
As already mentioned, it seems desirable to have a proof of the results of~\cite{EF} using modular theory. In the remainder of this section we provide some first steps in this direction.

\medskip

Given $m \geq 0$, consider the following real subspaces of $L^2(\bR^d)$:
\[
\cL_\varphi := \overline{\omega_m^{-1/2} \cD_\bR(B)}, \qquad \cL_\pi :=\overline{\omega_m^{1/2}\cD_\bR(B)}
\]
(closure in the $L^2$ norm), and the respective orthogonal projections $E_\varphi$, $E_\pi$, w.r.t. the real scalar product $\Re \langle \cdot, \cdot \rangle$. Since $\omega_m$ maps real functions to real functions, it is clear that $\cL_\varphi$ and $i\cL_\pi$ are real orthogonal. Moreover it is easy to check that 
\[
H^{-1/2}_\bR(B) = \omega_m^{1/2} \cL_\varphi, \qquad H^{1/2}_\bR(B) = \omega_m^{-1/2} \cL_\pi,
\]
and therefore $K_m = \cL_\varphi + i \cL_\pi$ and the real projection $E_{K_m} : L^2(\bR^d) \to K_m$ satisfies $E_{K_m} = E_\varphi - i E_\pi i$.

In the following, we will consider the symplectomorphism $Q$ in the case $m_1 = m  > 0$, $m_2 = 0$, which is the most interesting one. 

\begin{Lemma}
The symplectomorphism $Q : K_m \to K_0$ is bounded and there holds, on $K_m$,
\begin{equation}\label{eq:T}
Q^\dagger Q = \left( E_\varphi \frac{\omega_m}{\omega_0} E_\varphi + i E_\pi \frac{\omega_0}{\omega_m} E_\pi i\right)\bigg|_{K_m}.
\end{equation}
\end{Lemma}

\begin{proof}
The boundedness of $Q$ follows from the factoriality of $K_m$, $K_0$ (see the proof of Cor.~\ref{factor}), but it can also be proven directly as a consequence of~\cite[Prop.\ A.2]{FG}, cf.\ the proof of Thm.~\ref{thm:resolv} below. If $f, f' \in H_\bR^{-1/2} (B)$, $g,g' \in H_\bR^{1/2}(B)$, from the identity
\[\begin{split}
\Re \langle \omega_0^{-1/2} f + i\omega_0^{1/2} g, \omega_0^{-1/2} f'+i\omega_0^{1/2}g'\rangle &= \langle f, \omega_0^{-1} f'\rangle + \langle g, \omega_0 g'\rangle \\
&=  \langle \omega_m^{1/2} \omega_0^{-1}f, \omega_m^{-1/2} f'\rangle + \langle \omega_m^{-1/2} \omega_0 g, \omega_m^{1/2} g'\rangle\\
&= \Re \langle E_{K_m}(\omega_m^{1/2} \omega_0^{-1} f + i \omega_m^{-1/2} \omega_0 g),\omega_m^{-1/2} f'+i \omega_m^{1/2} g' \rangle,
\end{split}\]
there follows, for $\psi \in K_m$,
\[
Q^\dagger Q \psi = E_{K_m} \left( \frac{\omega_m}{\omega_0} \Re \psi + i \frac{\omega_0}{\omega_m} \Im \psi\right) = (E_\varphi - i E_\pi i)\left( \frac{\omega_m}{\omega_0} E_\varphi+ i \frac{\omega_0}{\omega_m} E_\pi i\right) \psi.
\]
Eq.~\eqref{eq:T} then follows by observing that the ranges of $i\frac{\omega_0}{\omega_m} E_\pi$ and $i\frac{\omega_m}{\omega_0} E_\varphi $ are made of purely imaginary functions, and therefore they are real orthogonal to $\cL_\varphi$ and $\cL_\pi$ respectively,
which entails $E_\varphi i \frac{\omega_0}{\omega_m} E_\pi = 0$ and $iE_\pi i \frac{\omega_m}{\omega_0} E_\varphi = 0$.
\end{proof}

There is another way to arrive to the same result, which is somewhat less direct but it has the advantage of making more apparent the connection with other related work, notably \cite{FG,LM}.
It is based on the observation that the map
\begin{equation}\label{eq:isoK}
H^{-1/2}(B) \oplus H^{1/2}(B) \ni (f,g) \mapsto 2^{-1/2}(\omega_m^{-1/2}f-i \omega_m^{1/2} g) \in K_m + i K_m
\end{equation}
defines a unitary equivalence between $H^{-1/2}(B) \oplus H^{1/2}(B)$ and $K_m + i K_m$, with its own Hilbert space structure.
Under this identification,
 it is not hard to see that, thanks to \cite[Section 3]{FG},
\begin{equation}
Q^\dagger Q  = \frac{1-\delta_m}{1+\delta_m}Q^{-1}\frac{1+\delta_0}{1-\delta_0}Q
= \begin{pmatrix} P_m^- \omega_m P_m^+ \chi_B \omega_0^{-1} & 0 \\ 0 & P_m^+ \omega_m^{-1} P_m^{-} \chi_B \omega_0 \end{pmatrix} \ 
\end{equation}
(on $H^{-1/2}(B) \oplus H^{1/2}(B)$),
where $P_m^\pm$ denotes the orthogonal projection from the global Sobolev space $H_m^{\pm1/2}(\bR^d)$ onto $H^{\pm1/2}(B)$ . Formula~\eqref{eq:T} is then obtained by observing that the restriction of~\eqref{eq:isoK} to $H^{-1/2}_\bR(B)\oplus H^{1/2}_\bR(B)$ is onto $K_m\subset L^2(\bR^d) = L^2_\bR(\bR^d) \oplus L^2_\bR(\bR^d)$, and is given by the direct sum of the (restriction to $H^{\pm 1/2}_\bR(B)$ of the) natural unitaries from $H_\bR ^{\pm 1/2} (\bR^d)$ onto $L_\bR^2(\bR^d)$ mapping $f$ into $\omega_m^{\pm 1/2} f$.

\medskip

We are now ready prove the Hilbert-Schmidt property of $1-Q^\dagger Q$. To this end, 
it is useful to observe that indeed 
$$E_\varphi = \omega_m^{-1/2} \chi \omega_m^{1/2} E_\varphi, \quad E_\pi = \omega_m^{1/2} \chi \omega_m^{-1/2} E_\pi$$ 
for any $\chi \in \cD_\bR(\bR^d)$ such that $\chi = 1$ on $B$. 

\begin{Theorem}\label{thm:HS}
For $d = 2,3$, the operator $1 - Q^\dagger Q$ is Hilbert-Schmidt on $K_m + i K_m$. Moreover, $\lim_{m \to 0^+} \| 1- Q^\dagger Q\|_2 = 0$.
\end{Theorem}

\begin{proof}
For the first claim, it suffices to show that $1 - Q^\dagger Q$ is Hilbert-Schmidt on $K_m$ equipped with the restriction of the scalar product of $K_m + i K_m$, which coincides (up to a multiplicative constant) with the real part of the standard scalar product of $L^2(\bR^d)$. Since this induces the same norm as the full scalar product of $L^2(\bR^d)$, we can  equivalently prove that 
$E_{K_m} (1-Q^\dagger Q) E_{K_m} = E_{K_m} - Q^\dagger Q E_{K_m}$ is Hilbert-Schmidt on $L^2 (\bR^d)$. Now,
\begin{equation}\label{eq:diffHS}
\begin{split}
E_{K_m} - Q^\dagger Q E_{K_m} & 
 = E_\varphi - i E_\pi i - \left(E_\varphi \frac{\omega_m}{\omega_0} E_\varphi + i E_\pi \frac{\omega_0}{\omega_m} E_\pi i \right) (E_\varphi - i E_\pi i) \\
 & = E_\varphi \left( 1 - \frac{\omega_m}{\omega_0} \right) E_\varphi - i E_\pi \left( 1 - \frac{\omega_0}{\omega_m} \right) E_\pi i
\end{split}
\end{equation}
If we prove that both the two terms in the last sum are Hilbert-Schmidt we are done.
Let us start with
\[E_\varphi \left( 1 - \frac{\omega_m}{\omega_0} \right) E_\varphi  
= E_\varphi \omega_m^{1/2} \chi \omega_m^{-1/2}  \left( 1 - \frac{\omega_m}{\omega_0} \right) \omega_m^{-1/2} \chi \omega_m^{1/2} E_\varphi \ . \]
We will actually prove that $\omega_m^{1/2} \chi \omega_m^{-1/2}  \left( 1 - \frac{\omega_m}{\omega_0} \right) \omega_m^{-1/2} \chi \omega_m^{1/2}$ is Hilbert-Schmidt by showing that its integral kernel
\[\omega_m(p)^{1/2} \left[ \int_{\bR^d} dq \, \hat \chi(p-q) \left( \frac 1 {\omega_m(q)}- \frac 1 {\omega_0(q)} \right) \hat \chi(q-k) \right] \omega_m(k)^{1/2}\]
 is in $L^2(\bR^{2d})$. In doing so, we will keep track of the dependency on $m$ of our estimates, in order to gain control on the $m \to 0$ limit of the Hilbert-Schmidt norm of this operator. In particular, we will use capital letters, without any index, to denote positive constants independent of $m$.
 Writing for brevity $F(q) = \frac 1 {\omega_0(q)} - \frac 1 {\omega_m(q)} = \frac {m^2} {\omega_0(q) \omega_m(q)(\omega_0(q) + \omega_m(q))}, q \in \bR^d$, we have to show that
 \begin{equation}\label{HSest}
 \int_{\bR^d} dp \, \omega_m(p) \int_{\bR^d} dq_1 \int_{\bR^d} dq_2 \, |\hat \chi (p-q_1)| \, |\hat \chi(p-q_2)| \left[ \int_{\bR^d} dk \, \omega_m(k) |\hat\chi(q_1 - k)| \, |\hat\chi(q_2 - k)| \right] F(q_1) F(q_2) 
 \end{equation}
is convergent.
Now, since $\hat \chi$ is a Schwarz function, the integral in the square brackets can be estimated, for sufficiently large $n$, by an $m$-independent constant times
\[ \begin{split}
\int_{\bR^d} dk \, \frac {\max\{1,m\}(1+|k|)} {(1 + |q_1 - k|)^n (1 + |q_2 - k|)^n}  &= \int_{\bR^d} dh \, \frac{\max\{1,m\}(1+|q_1 - h|)}{(1+ |h|)^n (1 + |q_2 - q_1 + h|)^n } \\
&\leq C\max\{1,m\}(|q_1| + 1) \ . 
\end{split}\]
Moreover, we have that $F(q)$ is bounded by $C_m |q|^{-1} (1+ |q|)^{-2}$ for a suitable constant $C_m > 0$, and by the inequality
\[
\frac{m^2(1+|q|)^2}{m^2+|q|^2} = \frac{1+|q|^2}{1+\frac{|q|^2}{m^2}} \leq 1 \quad \text{for }0 < m \leq 1,\,q \in \bR^d,
\]
we can take $C_m = 1 $ for $0 < m \leq 1$. This, together with the previous estimate, entails that the double integral in $q_1$ and $q_2$ can be majorised by a product of two integrals of the form
\begin{equation}\label{eq:intchi}
\int_{\bR^d} dq \, \frac{|\hat \chi(p-q)|}{|q| (1+|q|)^s}
\end{equation}
with $s=1$ and $s=2$, respectively.
Each of the latter integrals is in turn bounded, for $|p| > 1$ and $n > d+1$, by
\[
\begin{split}
\int_{\bR^d} \,  & \frac {dq}  {|q| (1+|q|)^s (1 + |p-q|)^n}   \\
& = \int_{|q| \leq |p|/2} \, \frac {dq} {|q| (1+|q|)^s (1 + |p-q|)^n} + \int_{|q| \geq |p|/2} \, \frac {dq} {|q| (1+|q|)^s (1 + |p-q|)^n}  \\
& \leq \frac 1 {(1 + |p|/2)^n} \int_{|q| \leq |p|/2} \, \frac {dq} {|q| (1+|q|)^s }+ \frac 2 {|p|  (1 + |p|/2)^s} \int_{|q| \geq |p|/2} \, \frac {dq} { (1 + |p-q|)^n} \\
&\leq C\left[\frac 1{(1+|p|)^{n-d+s}}+ \frac{1}{(1+|p|)^{s+1}}\right] \leq \frac{C'}{(1+|p|)^{s+1}},
\end{split}
\]
where in the last but one inequality we have exploited the fact that the first integral diverges for large $|p|$ like $|p|^{d-s}$, and the second one is bounded by the integral over all $\bR^d$, which is independent of $p$. This finally shows that the integral \eqref{HSest} is dominated by
\[
C'' C_m \max\{1,m\}\int_{\bR^d} dp \frac{\omega_m(p)}{(1+|p|)^5}  \leq C'' C_m \max\{1,m\}^2\int_{\bR^d}  \frac{dp}{(1+|p|)^4} < +\infty.
\]

Analogously, proving the Hilbert-Schmidt property for the second term in~\eqref{eq:diffHS} amounts to showing that
 \begin{equation}\label{HSest2}
 \int_{\bR^d} \frac{dp}{\omega_m(p)} \int_{\bR^d} dq_1 \int_{\bR^d} dq_2 \, |\hat \chi (p-q_1)| \, |\hat \chi(p-q_2)| \left[ \int_{\bR^d} dk \,\frac{ |\hat\chi(q_1 - k)| \, |\hat\chi(q_2 - k)| }{\omega_m(k)}\right] G(q_1) G(q_2) < + \infty,
 \end{equation}
 with $G(q) := \omega_m(q)-\omega_0(q)= \frac{m^2}{\omega_m(q)+\omega_0(q)}$. The integral in square brackets can now be bounded for large $n$ by an $m$-independent constant times
 \[\begin{split}
 \int_{\bR^d} \frac{dk}{\omega_m(k)(1+|q_1-k|)^n }&= \frac1{1+|q_1|} \int_{\bR^d} dh\frac{1+|q_1|}{\omega_m(q_1-h)(1+|h|)^n } \\
 &\leq  \frac1{1+|q_1|} \int_{\bR^d} dh\left[ \frac{1+|q_1-h|}{\omega_m(q_1-h)}\frac1{(1+|h|)^n }+\frac1m \frac{|h|}{(1+|h|)^n}\right] \\
 &\leq \frac{C}m \frac{1}{1+|q_1|},
 \end{split}\]
 where we used the straightforward estimate
 \[
 \frac{1+|k|}{\omega_m(k)} \leq \frac{1+m^2}{(m^2+m^4)^{1/2}}\leq \frac{C'}m, \qquad k \in \bR^d.
 \]
 Furthermore, $G(q)$ is bounded by $m C_m (1+|q|)^{-1}$ for a suitable constant $C_m > 0$, and by the estimate
 \[
 \frac{m^2(1+|q|)}{|q|+\omega_m(q)} \leq m \frac{1+|q|}{(1+|q|^2)^{1/2}}, \qquad 0 < m \leq 1, \, q \in \bR^d,
 \] 
 we see that we can assume that $C_m$ does not depend on $m$ for $0 < m \leq 1$.
 Therefore, reasoning as above we see that in place of~\eqref{eq:intchi} we have a product of two integrals of the form
 \[
 \int_{\bR^d} dq \, \frac{|\hat \chi(p-q)|}{(1+|q|)^s} \leq \frac{C''}{(1+|p|)^s}
 \]
 with $s=1$ and $s=2$, so that finally~\eqref{HSest2} is bounded by
 \[
 m C''' C^2_m \int_{\bR^d} \frac{dp}{\omega_m(p) (1+|p|)^3} \leq m C''' C^2_m \int_{\bR^d} \frac{dp}{|p| (1+|p|)^3}< +\infty,
 \]
 which concludes the proof of the first statement.
 
 Concerning the proof of the second statement, we first observe that, by what we saw above, the Hilbert-Schmidt norm of $1-Q^\dagger Q$ on $K_m+iK_m$ is proportional to the Hilbert-Schmidt norm of~\eqref{eq:diffHS} on $L^2(\bR^d)$. The fact that this latter norm vanishes for $m \to 0^+$ is then a straightforward application of the dominated convergence theorem, since the integrands in~\eqref{HSest}, \eqref{HSest2} vanish pointwise in this limit, and are bounded by integrable functions uniformly for $0 < m \leq 1$ by the above estimates.
\end{proof}

Note that the above proof does not work in $d=4$, which is compatible with the fact that the massive and massless vacua are not locally quasi-equivalent in this case, cf. \cite[Sec. 5]{BV2}.
\smallskip

As an immediate consequence, since $1-Q^\dagger Q = (1-|Q|)(1+|Q|)$ and $1+|Q|$ has a bounded inverse, $1-|Q|$ is Hilbert-Schmidt too.

Due to the unboundedness of $s_m$, Prop.~\ref{prop:sunbound}, the result just obtained is not sufficient to conclude that $Q$ induces a quasifree isomorphism of the local von Neumann algebras of the massive and massless Klein-Gordon field. On the other hand, it seems unlikely $1-Q^\dagger Q$ is trace class. Moreover, it seems that one lacks a sufficiently explicit knowledge of the second operator in~\eqref{eq:equiv} in order to prove its Hilbert-Schmidt property.

\medskip

It is a long standing expectation that $\delta_m$ is, in some sense, a small perturbation of $\delta_0$ (cf., e.g.,~\cite[Sec.\ 6]{SW}). A first quantitative version of this idea has been considered in \cite{FG}, where, analyzing the resolvents of modular operators, the continuity of the associated one-parameter unitary group w.r.t.\ to the mass was proven.
Here we point out another consequence of our analysis in this direction, namely the resolvent of the modular operator at mass $m$ is a perturbation in the Hilbert-Schmidt class of the resolvent of the modular operator at mass $0$. Furtermore, we show the continuity of the resolvents in the Hilbert-Schmidt norm.

To this end, 
first note that, for any given mass $m \geq 0$, all the resolvent operators
$R(\lambda,\delta_m) = (\delta_m - \lambda 1 )^{-1}$, $\lambda \in {\mathbb C} \setminus \sigma(\delta_m)$ leave $K_m + i K_m$ globally invariant, as they can be obtained as continuous functional calculi of $\frac{1}{1+\delta_m}$.
\begin{Theorem}\label{thm:resolv}
Let $m$ be positive and $Q: K_m \to K_0$ as above.
For every $\lambda \in {\mathbb C} \setminus [0,+\infty)$, the operator
\[
R(\lambda,\delta_m) - Q^{-1} R(\lambda,\delta_0) Q
\]
is Hilbert-Schmidt on $K_m + i K_m$ (w.r.t. the scalar product $\langle \cdot, \cdot \rangle_{s_m}$), and its Hilbert-Schmidt norm vanishes in the limit $m \to 0^+$.
\end{Theorem}

\begin{proof}
Consider first the case $\lambda_0 = -1$. There holds
 \[\begin{split}
R(-1,\delta_m) - Q^{-1} R(-1,\delta_0) Q &= \frac{1}{1+\delta_m} - Q^{-1} \frac{1}{1+\delta_0} Q = \frac12\left(\frac{1-\delta_m}{1+\delta_m}- Q^{-1}\frac{1-\delta_0}{1+\delta_0}Q\right) \\
&= -\frac12(1- Q^\dagger Q)Q^{-1}\frac{1-\delta_0}{1+\delta_0}Q,
\end{split}\]
where the second and third equalities follow respectively from the identities
\[
\frac1{1+\delta} = \frac12\left(1+\frac{1-\delta}{1+\delta}\right), \qquad Q^\dagger = \frac{1-\delta_m}{1+\delta_m} Q^{-1} \frac{1+\delta_0}{1-\delta_0},
\]
and the latter has been shown in the proof of Cor.~\ref{factor}. The
first claim then follows immediately from Thm.~\ref{thm:HS}. In order to prove the second claim it is then sufficient to show that the operator norms of $Q : K_m + i K_m \to K_0+iK_0$ and of $Q^{-1}: K_0 + i K_0 \to K_m + i K_m$ are uniformly bounded in $m$ for $m$ in a neighbourhood of $0$. This follows at once from the estimates in~\cite[Prop.\ A.2]{FG}, according to which, for $f \in H^{-1/2}_\bR(B)$, $g \in H^{1/2}_{\bR}(B)$,
\[\begin{split}
\| Q(\omega_m^{-1/2} f +i \omega_m^{1/2} g) \|^2 &= \|\omega_0^{-1/2} f +i \omega_0^{1/2} g\|^2 = \|f\|_{-1/2,0}^2+\|g\|_{1/2,0}^2 \\
&\leq c(m,B)^2\|f\|_{-1/2,m}^2+\|g\|_{1/2,m}^2 \leq c(m,B)^2 \|\omega_m^{-1/2} f +i \omega_m^{1/2} g\|^2, \\
\| Q^{-1}(\omega_0^{-1/2} f +i \omega_0^{1/2} g) \|^2 &= \|\omega_m^{-1/2} f +i \omega_m^{1/2} g\|^2 = \|f\|_{-1/2,m}^2+\|g\|_{1/2,m}^2 \\
&\leq \|f\|_{-1/2,0}^2+c(m,B)^2\|g\|_{1/2,0}^2 \leq c(m,B)^2 \|\omega_0^{-1/2} f +i \omega_0^{1/2} g\|^2, 
\end{split}\]
with $\sup_{0 < m \leq 1} c(m,B) < +\infty$.

Let now $\lambda \in {\mathbb C} \setminus [0,+\infty)$ be arbitrary. From the identity
\[
\left(1+\frac{\lambda+1}{\delta_m - \lambda}\right)\frac1{1+\delta_m} = \frac1{\delta_m-\lambda},
\]
it follows that defining $E=1+(\lambda+1)(\delta_m - \lambda 1)^{-1}$, $F=Q^{-1}[1+(\lambda+1)(\delta_0 - \lambda 1)^{-1} ]Q$, 
one has
\[
 E\Big(\frac{1}{1+\delta_m} - Q^{-1} \frac{1}{1+\delta_0} Q\Big)F =  \frac{1}{\delta_m-\lambda}F - E Q^{-1}\frac1{\delta_0-\lambda} Q = R(\lambda,\delta_m) - Q^{-1} R(\lambda,\delta_0) Q.
 \]
 Since, by functional calculus, the operator norms of $E$,$F$ are uniformly bounded in $m$, the statement follows from the above identity.
\end{proof}

We also remark that from the identity
\begin{multline*}
\left(\frac1{1+\delta_m} + Q^{-1}\frac1{1+\delta_0}Q\right)\left(\frac1{1+\delta_m}-Q^{-1}\frac1{1+\delta_0}Q\right) \\
= \left(\frac1{1+\delta_m}\right)^2- Q^{-1}\left(\frac1{1+\delta_0}\right)^2 Q+\left[ \frac1{1+\delta_0}, Q^{-1}\frac1{1+\delta_m}Q\right]
\end{multline*}
it follows that the commutator $\left[ \frac1{1+\delta_0}, Q^{-1}\frac1{1+\delta_m}Q\right]$ is Hilbert-Schmidt too. By a similar reasoning one can also show that
\[
\frac1{(1+\delta_0)^\alpha} Q^{-1} \frac1{1+\delta_m}Q - Q^{-1}\frac1{1+\delta_m} Q \frac1{(1+\delta_0)^\alpha}
\]
is Hilbert-Schmidt for all $\alpha >0$.

\subsection{Free scalar field in $d=1$}\label{subsec:lowdim}
We now turn back to the setting of \cite[Section 3]{CM3} to which we refer for notations, terminology and further details, and consider $H = L^2(\bR)$ with the usual scalar product. Denoting by $\cD_0(I)$ the space of smooth, complex valued functions $f$ with compact support in the interval $I \subset \bR$ such that $\int_I f = 0$, we define, for $I$ bounded, the real closed subspace of $H$
$$K_m(I)=\{\omega_m^{-1/2} \Re f + i \omega_m^{1/2} \Im f \ | \ f \in \cD_0(I) \}^{-\|\cdot\|} \ , \qquad m \geq 0.$$

\begin{Proposition}
For any $m>0$,  the subspace $K_m(I)$ is a standard subspace of $H$.
\end{Proposition}
\begin{proof}
It is enough to show that the vector $\Om \in e^H$ is cyclic for $\cCm(O_I)$. By \cite[Proposition 3.1(i)]{CM3}, the statement will follow from the fact that,
for any $I \Subset J$, $\overline{\cCm(I) \Om} = \overline{\cCm(J) \Om }$. 
To this end, consider open intervals $I_{n,k} \subset J$, $k=1,\dots,n$, of fixed length $\varepsilon_n$ such that $\varepsilon_n \to 0$ as $n \to +\infty$, $J = \bigcup_{k=1}^n I_{n,k}$ and $I_{n,k}\cap I_{n,k+1} \neq \emptyset$ for $k=1,\dots,n-1$. Then we pick functions $\varphi_k \in \cD(I_{n,k}\cap I_{n,k+1})$ with $\int \varphi_k = 1$, and given $f \in \cD_0(J)$ it is possible to find a partition of unity $\chi_1, \dots, \chi_n$ of the compact set $\supp f \subset J$ with $\supp \chi_k \subset I_{n,k}$. This entails $f = \sum_{k=1}^n f_k$ with $f_k \in\cD_0(I_{n,k})$ given by
\[
f_1 = \chi_1 f -\alpha_1 \varphi_1, \quad f_k = \sum_{j=1}^{k-1} \alpha_j\varphi_{k-1}+\chi_k f - \sum_{j=1}^k \alpha_j \varphi_k, \; k=2,\dots,n-1, \quad f_n = \sum_{j=1}^{n-1}\alpha_j \varphi_{n-1}+\chi_n f,
\] 
where $\alpha_j = \int \chi_j f$. This shows that $\cCm(O_J) = \bigvee_{k=1}^n \cCm(O_{I_{n,k}})$. Then, for $n \in \bN$ sufficiently large, one can find $\bar k$ such that $I_{n,\bar k} \Subset I$. Therefore if $\Phi \in (\cCm(O_I)\Om)^\perp$ then, by the Reeh-Schlieder argument, $\Phi \in \Big(\bigvee_{x \in {\bR}} \cCm(I_{n,\bar k} + x)\Om\Big)^\perp$ so that $\Phi \in (\bigvee_{k=1}^n\cCm(O_{I_{n,k}})\Om)^\perp = (\cCm(O_J)\Om)^\perp$.
\end{proof}

\begin{Proposition}
$K_0(I)$ is a standard subspace of the complex Hilbert space
\[
H_0=\{\omega_0^{-1/2} f + i \omega_0^{1/2} g \ | \ f, g \in \cD_0(\bR)\}^{-\|\cdot\|} \subset L^2(\bR) \ .
\]
\end{Proposition}

\begin{proof}
By a similar argument as in the above proposition, the net $\cC^{(0)}$ satisfies weak additivity, and then the Reeh-Schlieder property, on its cyclic Hilbert space  $\cK^{(0)} = \overline{\cC^{(0)}\Omega^{(0)}}$. Moreover, the operators on $e^{H_0}$
\begin{equation}\label{eq:W0}
W_0(f+i g) := W(\omega_0^{-1/2} f + i \omega_0^{1/2} g ), \qquad f, g \in \cD(\bR,\bR), \int f = 0 = \int g,
\end{equation}
generate the Weyl C*-algebra associated to the symplectic space $\cD_0(\bR)$ on which $\Omega \in e^{H_0}$ induces the vacuum state $\omega^{(0)}$ whose GNS representation defines $\cC^{(0)}$. We show that $\Omega$ is cyclic for the algebra generated by the operators~\eqref{eq:W0}. To this end, let $\Phi \in e^{H_0}$ be orthogonal to all vectors $W_0(f+ig) \Omega$. Then, for all $\lambda \in \bR$,
\[
\langle \Phi, W_0(\lambda(f+ig)) \Omega\rangle = e^{-\frac14\|\omega_0^{-1/2} f + i \omega_0^{1/2} g\|^2} \sum_{n=0}^{+\infty}\frac1{\sqrt{n!}} \left(\frac{i\lambda}{\sqrt{2}}\right)^n \langle \Phi_n, (\omega_0^{-1/2} f + i \omega_0^{1/2} g)^{\otimes n}\rangle = 0,
\]
where the power series converges for all $\lambda \in \bR$, so that
\[
\langle \Phi_n, (\omega_0^{-1/2} f + i \omega_0^{1/2} g)^{\otimes n}\rangle = 0, \qquad n \in \bN.
\]
By polarization and complex linearity this entails
\[
\langle \Phi_n,S_n\big((\omega_0^{-1/2} f_1 + i \omega_0^{1/2} g_1)\otimes \dots \otimes (\omega_0^{-1/2} f_n + i \omega_0^{1/2} g_n)\big)\rangle = 0, \qquad n \in \bN.
\]
for all $f_j, g_j \in \cD(\bR,\bC)$, $\int f_j = 0 = \int g_j$, $j=1,\dots, n$, i.e., $\Phi = 0$. Therefore, by GNS unicity, the algebras $\cC^{(0)}(O_I)$ on $\cK^{(0)}$ are unitarily equivalent to the second quantization algebras $A(K_0(I))$ on $e^{H_0}$, which proves the claim.
\end{proof}

The following result is a reformulation, in the present framework, of \cite[Thm.\ A.1]{CM3}.

\begin{Theorem}
The map $Q : K_m(I) \to K_0(I)$, $\omega_m^{-1/2} \Re f + i \omega_m^{1/2} \Im f \mapsto \omega_0^{-1/2} \Re f + i \omega_0^{1/2} \Im f$, is a symplectomorphism which induces a quasi free isomorphism between the corresponding second quantization algebras.
\end{Theorem}

\begin{proof}
It is clear that Thm.~\ref{main} and its proof remain valid, \emph{mutatis mutandis}, also for a symplectomorphism between standard subspaces $K_1$, $K_2$ of two different complex Hilbert spaces $H_1$, $H_2$. It is then sufficient to prove that $1-Q^\dagger Q$ is a trace class operator on $K_m(I)$. To this end, consider the map
\[
V_m : \overline{\cD_0(I)}^{\|\cdot\|_m} \to K_m(I), \quad f \mapsto \omega_m^{-1/2} \Re f + i \omega_m^{1/2} \Im f, \qquad m \geq 0.
\]
It is clear that $V_m$ is an orthogonal transformation between real Hilbert spaces and that $QV_m = V_0$. Therefore
\[
\Re \langle f, V_m^{-1}Q^\dagger Q V_m g\rangle_m = \Re \langle QV_m f, QV_m g \rangle  = \Re \langle f, g\rangle_0
\]
implying $V_m^{-1}Q^\dagger Q V_m= T$,  the operator defined in~\cite[App.\ A]{CM3}. The statement then follows by~\cite[Lemmas\ A.4 and A.6]{CM3}.
\end{proof}

\emph{Acknowledgements.} We thank Daniele Guido, Roberto Longo and Detlev Buchholz for useful discussions on the subject of this work. We acknowledge partial support by the INdAM-GNAMPA. R.C.\ is partially supported by Sapienza University of Rome (Progetti di Ateneo 2019, 2020). G.M.\ is partially supported by the ERC Advanced Grant 669240 \emph{QUEST}, the MIUR Excellence Department Project \emph{MatMod@TOV} awarded to
the Department of Mathematics, University of Rome ``Tor Vergata'', CUP E83C23000330006, and the University of Rome ``Tor Vergata'' funding \emph{OAQM}, CUP E83C22001800005.



\end{document}